\documentclass[12pt]{article}

\usepackage{fullpage}
\usepackage[utf8]{inputenc}
\usepackage{amsmath}
\usepackage{amssymb}
\usepackage{amsthm}
\usepackage{tikz}
\usetikzlibrary{shapes}

\theoremstyle{plain}
\newtheorem{thm}{Theorem}
\newtheorem{lem}{Lemma}
\newtheorem{prop}{Proposition}
\newtheorem{cor}{Corollary}

\theoremstyle{definition}

\theoremstyle{remark}

\newcommand{\G}{\mathrm{G}}
\newcommand{\St}{\mathrm{S}}
\newcommand{\T}{\mathrm{T}}

\author{Jonathan Chappelon\footnote{Laboratoire de Math\'ematiques Pures et Appliqu\'ees, Universit\'e du Littoral C\^ote d'Opale.}\quad and\quad Akihiro Matsuura\footnote{School of Science and Engineering, Tokyo Denki University.}}
\date{September 2, 2011}
\title{On generalized Frame-Stewart numbers}

\begin{document}

\maketitle

\begin{abstract}
For the multi-peg Tower of Hanoi problem with $k \ge 4$ pegs, so far the best solution is obtained by the Stewart's algorithm \cite{Stewart2} based on the the following recurrence relation:
$$
\St_k(n)=\min_{1 \le t \le n} \bigl\{2 \cdot \St_k(n-t) + \St_{k-1}(t)\bigr\}, \ \ \St_3(n) = 2^n - 1.
$$
In this paper, we generalize this recurrence relation to
$$
\G_k(n) = \min_{1\le t\le n}\bigl\{ p_k\cdot \G_k(n-t) + q_k\cdot \G_{k-1}(t) \bigr\}, \ \ \G_3(n) = p_3\cdot \G_3(n-1) + q_3,
$$
for two sequences of arbitrary positive integers $(p_i)_{i \ge 3}$ and $(q_i)_{i \ge 3}$ and we show that the sequence of differences $(\G_k(n)- \G_k(n-1))_{n \ge 1}$ consists of numbers of the form $(\prod_{i=3}^{k}q_i) \cdot (\prod_{i=3}^{k}{p_i}^{\alpha_i})$, with $\alpha_i\ge 0$ for all $i$, arranged in nondecreasing order. We also apply this result to analyze recurrence relations for the Tower of Hanoi problems on several graphs.
\end{abstract}
\section{Introduction}

The Tower of Hanoi problem was introduced by \'{E}douard Lucas in 1883 \cite{Lucas} for the case of 3 pegs and $n$ disks of different sizes. Initially, $n$ disks are placed on one of the 3 pegs with the largest at the bottom. Then, at each time one of the topmost disks is moved to a peg with a larger disk on the top or to an empty peg. The goal of the problem is to transfer all the disks from the initial peg to the peg of destination with the minimum number of moves. A simple recursive argument shows that $2^n-1$ moves are necessary and sufficient to carry out this task. This Tower of Hanoi problem was then extended to the case of 4 pegs by Dudeney in 1907 \cite{Dude} and to arbitrary $k \ge 3$ pegs by Stewart in 1939 \cite{Stewart1}. In 1941, Frame \cite{Frame} and Stewart \cite{Stewart2} independently proposed algorithms which achieve the same numbers of moves for the $k$-peg Tower of Hanoi problem with $k \ge 4$ pegs. Klav\v{z}ar et al.\cite{Klav1} showed that seven different approaches to the $k$-peg Tower of Hanoi problem, including those by Frame and Stewart, are all equivalent, that is, achieve the same numbers of moves. Thus, these numbers are called the {\it Frame-Stewart numbers} \cite{Klav2}.

Somewhat surprisingly, the optimal solution for the multi-peg Tower of Hanoi problem with $k \ge 4$ pegs is not known yet. So far, the best upper bounds are achieved by the Frame-Stewart numbers and the best lower bounds are obtained by Chen et al.\cite{Chen}. Since the upper bounds are believed to be optimal, they are called the ``presumed optimal'' solution.

The Stewart's recursive algorithm for the $k$-peg Tower of Hanoi problem is summarized as follows. For integer $t$ such that $1\leq t\leq n$,
\begin{enumerate}
\item recursively transfer a pile of $n-t$ smallest disks from the first peg to a temporary peg using $k$ pegs;
\item transfer the remaining pile of $t$ largest disks from the first peg to the final peg using $k-1$ pegs, ignoring the peg occupied by the $n-t$ smallest disks;
\item recursively transfer the pile of $n-t$ smallest disks from the temporary peg to the final peg using $k$ pegs.
\end{enumerate}
The algorithm chooses the integer $t$ such that the number of moves $2 \cdot \St_k(n-t) + \St_{k-1}(t)$ is minimized. Thus, the Frame-Stewart numbers $\St_k(n)$ satisfy the following recurrence relations:
$$\St_k(n) = \min_{1 \le t \le n} \bigl\{2 \cdot \St_k(n-t) + \St_{k-1}(t)\bigr\}, \mbox{ for } n \ge 1, \ k \ge 4,$$
$$\St_3(n) = 2^n - 1, \mbox{ for } n \ge 1, \mbox{ and } \St_k(0) = 0, \mbox{ for } k \ge 3.  $$
When $k=4$ for instance, $\St_4(n)$ is obtained by the following simple formula:
$$
\St_4(n) - \St_4(n-1) = 2^{i-1}, \mbox{ for } \binom{i}{2} < n \le \binom{i+1}{2},
$$
where $\binom{i}{2}$ is the binomial coefficient equal to $i(i-1)/2$. In the general case $k \ge 4$, $\St_k(n)$ is obtained by several different approaches, e.g., \cite{Frame, Klav1, Klav2, Majumdar, Stewart2}.

In \protect\cite{Mats}, the following general recurrence relation was considered to clarify the combinatorial structure latent in the recurrence relation for $\St_k(n)$ and to cope with the recurrence relations for the Tower of Hanoi {\it on graphs} in which pegs are placed on vertices of a given graph and disks are only moved along the edges:
$$
\T(n) = \min_{1\le t \le n}
\bigl\{\alpha \cdot \T(n-t) +  \beta \cdot (2^t - 1)  \bigr\}, \mbox{ for } n \ge 1, \mbox{ and } \T(0) = 0,
$$
where $\alpha$ and $\beta$ are arbitrary positive integers. It was shown that the sequence of differences $(\T(n) - \T(n-1))_{n \ge 1}$ consists of numbers of the form $\beta \cdot 2^i\cdot\alpha^j$, with $i,j \ge 0$, arranged in nondecreasing order. When $\alpha = 3$, $2^i \cdot \alpha^j$ increases as $1,2,3,2^2,2\cdot3,2^3,3^2,2^2\cdot3,2^4,2\cdot3^2, \cdots$. These numbers are called ``3-smooth numbers''\cite{Sloane} and have been studied extensively in number theory, in relation to the distribution of prime numbers \cite{Hardy} and to new number representations \cite{Bleck,Erd}. The formulation and analysis of $\T(n)$, however, has some defects such that (i) it is only focused on the 4-peg case with no consideration for the general case $k \ge 3$; and (ii) even in the 4-peg case, term $2^i \cdot\alpha^j$ consists of constant 2 and parameter $\alpha$, which might admit further generalization.

In this paper, we fully generalize the recurrence relations for the previous $\St_k(n)$ and $T(n)$ and obtain the exact formulas. Namely, we define the following recurrence relations for two sequences of arbitrary positive integers $\left(p_i\right)_{i \ge 3}$ and $\left(q_i\right)_{i \ge 3}$:
$$
\G_k(n) = \min_{1 \le t \le n}\bigl\{ p_k\cdot \G_k(n-t) + q_k\cdot \G_{k-1}(t) \bigr\},\ \text{for}\ n\ge 1,\ k\ge 4,
$$
$$
\G_3(n) = p_3\cdot \G_3(n-1)+q_3,\ \text{for}\ n\ge 1, \mbox{ and } \G_k(0) = 0,\ \text{for}\ k\ge 3.
$$
Then, we show that the sequence of differences $(\G_k(n)- \G_k(n-1))_{n \ge 1}$ consists of numbers of the form $(\prod_{i=3}^{k}q_i) \cdot (\prod_{i=3}^{k}{p_i}^{\alpha_i})$, with $\alpha_i\ge 0$ for all $i$, arranged in nondecreasing order. In other words, we show the following theorem.

\begin{thm}\label{thm1}
For every positive integer $n$ and for two sequences of arbitrary positive integers $\left(p_i\right)_{i \geq 3}$ and $\left(q_i\right)_{i \geq 3}$, we have
$$
\G_k(n) = q\cdot\sum_{j=1}^{n}u^k_j
$$
where $q=\prod_{i=3}^{k}q_i$ and $u^k_j$ is the $j$th term of the sequence $\left(u^k_j\right)_{j\geq1}$ of integers $\prod_{i=3}^{k}{p_i}^{\alpha_i}$, with $\alpha_i\geq0$ for all $i$, arranged in nondecreasing order.
\end{thm}
We call $\G_k(n)$ the {\it generalized Frame-Stewart numbers}. Note that $\G_k(n)$ is equal to $\St_k(n)$ when $(p_i, q_i) = (2, 1)$ for all $i \ge 3$ and $\G_4(n)$ is equal to $\T(n)$ when $(p_3, q_3) = (2, 1)$ and $(p_4, q_4) = (\alpha, \beta)$.
\par The remaining of the paper is organized as follows. In Section~2, we show some basic properties of the sequence $\left(u^k_j\right)_{j\geq 1}$ defined from $\left(p_i\right)_{i \geq 3}$. In Section~3, we prove Theorem~\ref{thm1}, the main result of this paper. In Section~4, application of these numbers in obtaining upper bounds of the number of moves for the Tower of Hanoi problem on several graphs is provided.

\section{Basic results on smooth number sequences}

Let $\left(p_i\right)_{i\geq3}$ be a sequence of positive integers. We consider the sequence $\left(u^k_j\right)_{j\geq1}$ of all the integers of the form $\prod_{i=3}^{k}{p_i}^{\alpha_i}$, where $\alpha_i\geq0$ for all $i$, arranged in nondecreasing order. For instance, for $(p_3,p_4)=(2,2)$ and $(p_3,p_4)=(2,3)$, the first few terms of $(u^4_j)_{j\geq1}$ are $(1,2,2,2^2,2^2,2^2,2^3,\cdots)$ and $(1,2,3,2^2,2\cdot3,2^3,3^2,\cdots)$, respectively. When there is some $i_0$ such that $p_{i_0}$ is equal to $1$, then by definition $\left(u^k_j\right)_{j\geq1}$ is the constant sequence of $1$'s, for every $k\geq i_0$. We note that $\left(u^k_j\right)_{j\geq1}$ is closely related to {\it smooth numbers} which have been explored extensively in number theory. A positive integer is called {\it $B$-smooth} if none of its prime factors are greater than a positive integer $B$. The sequence $\left(u^k_j\right)_{j\geq1}$ then consists of $B$-smooth numbers for $B = \max_{3 \leq i \leq k}\left\{p_i\right\}$.

In this section, we restrict to the case where all the $p_i$'s are greater than $1$ and prove a simple lemma on a certain ``recursive'' structure of the smooth number sequence $\left(u^k_j\right)_{j\geq1}$, which will be used to prove Theorem~\ref{thm1} in the next section.

\begin{lem}\label{lem1}
Let $k\geq4$ and let $\left(f_k(j)\right)_{j\geq1}$ be the sequence of positive integers defined by $f_k(1)=1$ and $f_k(j)=\min\left\{l>f_k(j-1) \ \middle| \ u^k_l=u^{k-1}_{j}\right\}$ for $j\geq2$. Then, for every integer $n$ such that $f_k(j)<n<f_k(j+1)$, we have $u^k_n = p_k\cdot u^k_{n-j}$.
\end{lem}

\begin{proof}
If $f_k(j+1)=f_k(j)+1$, then the lemma is trivial. Suppose now that $f_k(j+1)-f_k(j)\geq2$ and let $n$ be a positive integer such that $f_k(j)<n<f_k(j+1)$. First, consider a term $\prod_{i=3}^{k}{{p_i}^{\alpha_i}}$ of the sequence $(u^k_l)_{l\ge1}$. If $\alpha_k=0$, then $\prod_{i=3}^{k}{{p_i}^{\alpha_i}}=\prod_{i=3}^{k-1}{{p_i}^{\alpha_i}}$ belongs to $(u^k_{f_k(l)})_{l\geq1}$ by definition of $(f_k(l))_{l\geq1}$. Otherwise, if $\alpha_k\ge1$, then $\prod_{i=3}^{k}{{p_i}^{\alpha_i}}=p_k\cdot\left({p_k}^{\alpha_k-1}\cdot\prod_{i=3}^{k-1}{{p_i}^{\alpha_i}}\right)$ belongs to $(p_k\cdot u^k_l)_{l\geq1}$. Now, since $f_k(j)<n<f_k(j+1)$, it follows that $u^k_{f_k(j)}\le u^k_n < u^k_{f_k(j+1)}$ by the growth of the sequence $(u^k_l)_{l\ge1}$. We deduce that
$$
\left\{ u^k_l\ \middle|\ 1\leq l\leq n\right\} \bigcap \left\{u^k_{f_k(l)}\ \middle|\ l\geq1 \right\} = \left\{u^k_{f_k(l)}\ \middle|\ 1\leq l\leq j \right\}.
$$
Therefore, since a term of $\left(u^k_l\right)_{l\geq1}$ belongs to $\left(u^k_{f_k(l)}\right)_{l\geq1}$ or to $\left(p_k\cdot u^k_l\right)_{l\geq1}$, we obtain the following decomposition
$$
\left\{ u^k_l\ \middle|\ 1\leq l\leq n \right\} 
= \left\{ u^k_{f_k(l)}\ \middle|\ 1\leq l\leq j \right\} \bigcup \left\{ p_k\cdot u^k_l\ \middle|\ 1\leq l\leq n-j\right\}.
$$
This decomposition with the maximality of $u^k_n$ leads to
$$
u^k_n
\begin{array}[t]{l}
= \max\left\{u^k_l\ \middle|\ 1\leq l\leq n\right\} \\[2ex]
= \max\left\{ \max\left\{ u^k_{f_k(l)}\ \middle|\ 1\leq l\leq j \right\} , \max\left\{ p_k\cdot u^k_{l}\ \middle|\ 1\leq l\leq n-j \right\} \right\} \\[2ex]
= \max\left\{ u^k_{f_k(j)}\ ,\ p_k\cdot u^k_{n-j} \right\}.
\end{array}
$$
Since the hypothesis $f_k(j)<n<f_k(j+1)$ implies that $u^k_n$ belongs to $\left(p_k\cdot u^k_l\right)_{l\geq1}$, this completes the proof that $u^k_n=p_k\cdot u^k_{n-j}$.
\end{proof}

\section{Proof of Theorem~\ref{thm1}}

Let $\G_k^1(n)$ denotes the special case of $\G_k(n)$ associated with arbitrary sequence $\left(p_i\right)_{i\geq3}$ and with the constant sequence $\left(q_i\right)_{i\geq3}$ with $q_i=1$ for $i \geq 3$. There exists a simple relationship between numbers $\G_k(n)$ and $\G_k^1(n)$.

\begin{prop}\label{prop2}
For every nonnegative integer $n$ and for every sequence of integers $\left(q_i\right)_{i\geq3}$, we have
$$
\G_k(n) = q\cdot\G_k^1(n),
$$
where $q=\prod_{i=3}^{k}q_i$.
\end{prop}

\begin{proof}
By double induction on $k\geq3$ and $n\geq0$. For $k=3$, we can prove by simple induction on $n$ that $\G_3(n)=q_3\cdot\G_3^1(n)$ for all $n$. For $n=0$, we have $\G_k(0)=q\cdot\G_k^1(0)=0$ for all $k$. Suppose now that the result is true for $k-1$ and all $n\geq0$, and for $k$ and all $l\leq n-1$. By the recursive definition of $\G_k(n)$ and by the assumption of induction, we obtain
$$
\G_k(n) \begin{array}[t]{l}
= \displaystyle\min_{1\leq t\leq n}\left\{ p_k\cdot \G_k(n-t) + q_k\cdot\G_{k-1}(t)\right\} \\[2ex]
= \displaystyle\min_{1\leq t\leq n}\left\{ p_k\cdot \prod_{i=3}^{k}q_i\cdot\G_k^1(n-t) + q_k\cdot\prod_{i=3}^{k-1}q_i\cdot\G_{k-1}^1(t)\right\} \\[2ex]
= \displaystyle\prod_{i=3}^{k}q_i\cdot\min_{1\leq t\leq n}\left\{ p_k\cdot \G_k^1(n-t) + \G_{k-1}^1(t)\right\} \\[3ex]
= q\cdot\G_k^1(n).
\end{array}
$$
\end{proof}

By Proposition~\ref{prop2}, it is sufficient to prove Theorem~\ref{thm1} for $\G_k^1(n)$ instead of $\G_k(n)$. Now, we show at which argument
$
\G_k^1(n) =
 \displaystyle\min_{1\leq t\leq n}\left\{ p_k\cdot \G_k^1(n-t) + \G_{k-1}^1(t)\right\}
$ takes its minimum.

\begin{lem}\label{lem2}
Let $n$ be a positive integer. Suppose that $p_i > 1$ for all $3 \leq i \leq k$. Suppose also that $\Delta\G_{i}^1(l)=\G_{i}^1(l)-\G_{i}^1(l-1)=u^i_l$ for $3\leq i\leq k-1$ and $l\geq1$ and that $\Delta\G_k^1(l)=u^k_l$ for $1\leq l\leq n-1$. Let $j$ be the integer such that $f_k(j)\leq n<f_k(j+1)$. Then, for $1\leq t\leq n$, $\G_{k,n}^1(t)=p_k\cdot\G_k^1(n-t)+\G_{k-1}^1(t)$ takes its minimum at $t=j$.
\end{lem}

\begin{proof}
Since
$$
\G_{k,n}^1(t+1) - \G_{k,n}^1(t)
\begin{array}[t]{l}
= p_k\cdot\G_k^1(n-t-1)+\G_{k-1}^1(t+1) - p_k\cdot\G_k^1(n-t) - \G_{k-1}^1(t)\\[2ex]
= -p_k\cdot(\G_k^1(n-t)-\G_k^1(n-t-1)) + (\G_{k-1}^1(t+1)-\G_{k-1}^1(t))\\[2ex]
= -p_k\cdot\Delta\G_k^1(n-t) + \Delta\G_{k-1}^1(t+1)
\end{array}
$$
for every $1\leq t\leq n-1$, it follows by hypothesis that
$$
\G_{k,n}^1(t+1) - \G_{k,n}^1(t) = -p_k\cdot u^k_{n-t} + u^{k-1}_{t+1}\quad \text{for}\quad 1\leq t\leq n-1.
$$

First, when $1 \leq t\leq j-1$, the growth of the sequences $\left(u^k_l\right)_{l\geq1}$ and $\left(u^{k-1}_l\right)_{l\geq1}$ yields the following inequalities
$$
u^k_{n-t} \geq u^k_{n-j+1} \geq u^k_{f_k(j)-j+1}, \quad u^{k-1}_{t+1}\leq u^{k-1}_j=u^k_{f_k(j)}.
$$
Let $m=\min\left\{ l\geq0  \ \middle| \  f_k(j+l+1)-f_k(j+l)\geq2 \right\}$. Such $m$ always exists. By definition of $f_k(j+l)$, we have $f_k(j+l)=f_k(j)+l$ for $0\leq l\leq m$ and $f_k(j+m)<f_k(j)+m+1<f_k(j+m+1)$. So we deduce from Lemma~\ref{lem1} that
$$
u^k_{f_k(j)+m+1} = p_k\cdot u^k_{(f_k(j)+m+1)-(j+m)} = p_k\cdot u^k_{f_k(j)-j+1}.
$$
Thus,
$$
\G_{k,n}^1(t+1)-\G_{k,n}^1(t)  = -p_k\cdot u^k_{n-t} + u^{k-1}_{t+1}\leq -u^k_{f_k(j)+m+1} + u^k_{f_k(j)} \leq 0
$$
for $1\leq t\leq j-1$. Therefore, $\G_{k,n}^1(t)\geq \G_{k,n}^1(j)$ for all $1\leq t\leq j$.

Similarly, when $j\leq t\leq n-1$, we have
$$
u^k_{n-t} \leq u^k_{n-j} \leq u^k_{f_k(j+1)-j-1},\quad u^{k-1}_{t+1}\geq u^{k-1}_{j+1} = u^k_{f_k(j+1)}.
$$
Let $m=\min\left\{l\geq0  \ \middle| \  f_k(j-l+1)-f_k(j-l)\geq2 \right\}$. If such $m$ does not exist, then $n=f_k(j)=j$ and we already know that $\G_{k,n}^1(t)$ takes its minimum at $t=j$. Suppose now that the integer $m$ exists. By definition of $f_k(j-l+1)$, we have $f_k(j-l+1)=f_k(j+1)-l$ for $0\leq l\leq m$ and $f_k(j-m)<f_k(j+1)-m-1<f_k(j-m+1)$. So we deduce from Lemma~\ref{lem1} that
$$
u^k_{f_k(j+1)-m-1} = p_k\cdot u^k_{(f_k(j+1)-m-1)-(j-m)} = p_k\cdot u^k_{f_k(j+1)-j-1}.
$$
Thus,
$$
\G_{k,n}^1(t+1)-\G_{k,n}^1(t) = -p_k\cdot u^k_{n-t} + u^{k-1}_{t+1} \geq -u^k_{f_k(j+1)-m-1} + u^k_{f_k(j+1)} \geq 0
$$
for $j\leq t\leq n-1$. Therefore, $\G_{k,n}^1(t)\geq\G_{k,n}^1(j)$ for all $j\leq t\leq n$.

Consequently, $\G_{k,n}^1(t)$ takes its minimum at $t = j$.
\end{proof}

We are now ready to prove the main result of this paper.

\begin{proof}[Proof of Theorem~\ref{thm1}]
From Proposition~\ref{prop2}, it is sufficient to prove that
$$
\G_k^1(n) = \sum_{j=1}^{n}u^k_j
$$
for every positive integer $n$. We divide into different cases depending on the values of the terms of the sequence $\left(p_i\right)_{i\geq3}$.
\par\textbf{Case 1:} if $p_i>1$ for all $3\leq i\leq k$. We proceed by double induction on $k\geq3$ and $n\geq1$. For $k=3$, it is clear that $\G_3^1(1)=1$ and, by induction on $n\geq1$, that $\Delta \G_3^1(n) = \G_3^1(n)-\G_3^1(n-1) = p_3\cdot (\G_3^1(n-1) - \G_3^1(n-2)) = p^{n-1}_3=u^3_n$ for all $n \geq 2$. It is also clear that, for arbitrary $k$, $\G_k^1(1) = 1 = u^k_1$. Now assume that $\Delta\G_{i}^1(l)=u^i_l$ for all $3\leq i\leq k-1$ and all $l\geq1$ and that $\Delta\G_k^1(l)=u^k_l$ for all $1\leq l\leq n-1$. Then, we show that $\Delta\G_k^1(n)=u^k_n$ holds. For $n$, there exists some $j\geq1$ such that $f_k(j)\leq n<f_k(j+1)$. It is divided into two subcases: when $n=f_k(j)$ (Subcase 1.1) and when $f_k(j)<n<f_k(j+1)$ (Subcase 1.2).
\par \textbf{Subcase 1.1:} for $n=f_k(j)$. We obtain
$$
\Delta\G_k^1(n)\begin{array}[t]{l}
= \G_k^1(f_k(j))-\G_k^1(f_k(j)-1)\\[1.5ex]
= \G^1_{k,f_k(j)}(j)-\G^1_{k,f_k(j)-1}(j-1)\quad (\text{since}\ f_k(j-1)\leq f_k(j)-1<f_k(j) \text{ and by Lemma~\ref{lem2}})\\[1.5ex]
= p_k\cdot\left(\G_k^1(f_k(j)-j)-\G_k^1((f_k(j)-1)-(j-1))\right) + \left(\G_{k-1}^1(j)-\G_{k-1}^1(j-1)\right)\\[1.5ex]
= \Delta\G_{k-1}^1(j)\\[1.5ex]
= u^{k-1}_j\quad (\text{by\ assumption\ of\ induction})\\[1.5ex]
= u^k_{f_k(j)}\quad (\text{by\ definition\ of}\ f_k(j))\\[1.5ex]
= u^k_n.
\end{array}
$$
Thus, the proof is shown in this case.
\par \textbf{Subcase 1.2:} for $f_k(j)<n<f_k(j+1)$. We obtain
$$
\Delta\G_k^1(n)\begin{array}[t]{l}
= \G_k^1(n)-\G_k^1(n-1)\\[1.5ex]
= \G_{k,n}^1(j)-\G^1_{k,n-1}(j)\quad (\text{since}\ f_k(j)\leq n-1<f_k(j+1)\text{ and by Lemma~\ref{lem2}})\\[1.5ex]
= p_k\cdot\left(\G_k^1(n-j)-\G_k^1(n-1-j)\right) + \left(\G_{k-1}^1(j)-\G_{k-1}^1(j)\right)\\[1.5ex]
= p_k\cdot\Delta\G_k^1(n-j)\\[1.5ex]
= p_k\cdot u^k_{n-j}\quad (\text{by\ assumption\ of\ induction})\\[1.5ex]
= u^k_n\quad (\text{by\ Lemma~\ref{lem1}}).
\end{array}
$$
Thus, the proof is shown in this case, and this completes the proof of Case~1.
\par\textbf{Case 2:} if $p_i=1$ for some $3\leq i\leq k$. Let $m=\min\left\{3\leq i\leq k\ \middle|\ p_i=1 \right\}$. It is further divided into two subcases: when $k=m$ (Subcase 2.1) and when $k>m$ (Subcase 2.2).
\par\textbf{Subcase 2.1:} for $k=m$. If $k=m=3$, then $p_3=1$. In this case, it is clear that $\G_3^1(n)=n$ for all $n\geq1$. If $k=m\geq4$, that is, if $p_{k}=1$ and $p_i>1$ for all $3\leq i\leq k-1$, we proceed by induction on $n\geq1$. For $n=1$, we have $\G_k^1(1)=1$. Then assume that $\G^1_{k}(l)=l$ for $1\leq l\leq n-1$. By definition,
$$
\G^1_{k}(n) = \min_{1\leq t\leq n}\left\{\G^1_{k}(n-t)+\G^1_{k-1}(t)\right\} = \min_{1\leq t\leq n}\left\{(n-t)+\G^1_{k-1}(t)\right\}.
$$
Since $p_i>1$ for all $3\leq i\leq k-1$, we know that $\G^1_{k-1}(l)=\sum_{j=1}^{l}u^{k-1}_j$ for $l\geq1$ from Case~1. It is clear that $u^{k-1}_j\geq1$ for all $1\leq j\leq l$. Therefore, we have $\G^1_{k-1}(l)\geq l$ for $l\geq1$. So $\G^1_{k,n}(t)=(n-t)+\G^1_{k-1}(t)$ takes its minimum at $t=1$ and $\G^1_{k}(n)=(n-1)+1=n$ as announced.
\par\textbf{Subcase 2.2:} for $k> m$. We proceed by double induction on $k\geq m$ and $n\geq1$. We know that $\G^1_m(l)=l$ for all $l\geq1$ from Subcase~2.1. We also know that $\G^1_i(1)=1$ for all $i\geq3$.
Now, assume that $\G_{k-1}^1(l)=l$ for all $l\geq1$ and that $\G^1_{k}(l)=l$ for all $1\leq l\leq n-1$. We obtain
$$
\G^1_{k}(n) = \min_{1\leq t\leq n}\left\{\G^1_{k}(n-t)+\G_{k-1}^1(t)\right\} = \min_{1\leq t\leq n}\left\{(n-t)+t\right\} = n.
$$
This concludes the proof of Case~2, and thus the proof of Theorem~\ref{thm1}.
\end{proof}

\begin{cor}
Let $k\geq4$ and $j\geq1$. For every integer $n$ such that $f_k(j)\leq n<f_k(j+1)$,
$$
\G_k(n) = p_k\cdot\G_k(n-j)+q_k\cdot\G_{k-1}(j).
$$
\end{cor}

\begin{proof}
 From Proposition~\ref{prop2}, Theorem~\ref{thm1} and Lemma~\ref{lem2}.
\end{proof}

We end this section in considering the special case where $p_i=p\geq1$ for all $i$.

\begin{prop}\label{prop3}
Let $p_i=p\geq1$ for all $3\leq i\leq k$. Then, for all integers $j\geq0$ and $n\geq1$ such that
$$
\displaystyle\binom{k+j-3}{k-2} < n \leq \binom{k+j-2}{k-2},
$$
$u^k_n=p^j$ and $\G_k^1(n)$ can be computed as follows:
$$
\G_k^1(n) = \sum_{m=0}^{j-1}\binom{k+m-3}{k-3}p^m + \left(n-\binom{k+j-3}{k-2}\right)p^{j}.
$$
\end{prop}

\begin{proof}
Let $j$ be a nonnegative integer. First, we can determine $C_j$ the number of values of $n$ such that $u^k_n=p^j$. Then, since $C_j$ corresponds to the number of ways to distribute $j$ identical balls into $k-2$ distinct urns or the number of ways of partitioning $j$ into $k-2$ ordered non-negative summands, we have
$$
C_j = \binom{(k-2)+j-1}{(k-2)-1} = \binom{k+j-3}{k-3}.
$$
Now let $D_j$ be the number of values of $n$ such that $u^k_n<p^j$. Here we have
$$
D_j = \sum_{m=0}^{j-1}C_m = \sum_{m=0}^{j-1}\binom{k+m-3}{k-3}=\binom{k+j-3}{k-2}.
$$
It follows that $u^k_n=p^j$ exactly when $D_j<n\leq D_j+C_j=D_{j+1}$, that is, when
$$
\binom{k+j-3}{k-2} < n \leq \binom{k+j-3}{k-2} + \binom{k+j-3}{k-3} = \binom{k+j-2}{k-2}
$$
as claimed. This leads to the equality, for $D_j<n\leq D_{j+1}$,
$$
\G_k^1(n) = \sum_{m=0}^{j-1}C_m\cdot p^m + (n-D_j)\cdot p^j = \sum_{m=0}^{j-1}\binom{k+m-3}{k-3}p^m + \left(n-\binom{k+j-3}{k-2}\right)p^{j}.
$$
\end{proof}

\section{Application: the Tower of Hanoi on graphs}

Let $G=(V,E)$ be a simple graph with the set of vertices $V=\{v_1,\ldots,v_k\}$ and the set of edges $E$. A $k$-peg Tower of Hanoi problem can be considered on $G$: the $k$ pegs are placed on the vertices $v_1,\ldots,v_k$ and transfer of disks is allowed between the pegs $v_i$ and $v_j$ only if there is an edge between $v_i$ and $v_j$. The original $k$-peg Tower of Hanoi problem then corresponds to the Tower of Hanoi problem on the complete graph $\mathrm{K}_k$. The cases of $k=3$ and $k=4$ are illustrated in Figure~\ref{fig1}.

\begin{figure}[!h]
\begin{center}
\begin{tabular}{cccccc}

\begin{tikzpicture}
\node (A) at (0,0) [circle,draw] {$1$};
\node (B) at (2,0) [circle,draw] {$2$};
\node (C) at (1,1.73205081) [circle,draw] {$3$};
\draw (A) -- (B) -- (C) -- (A);
\end{tikzpicture}
& & & & &
\begin{tikzpicture}
\node (A) at (0,0) [circle,draw] {$1$};
\node (B) at (2,0) [circle,draw] {$2$};
\node (C) at (2,2) [circle,draw] {$3$};
\node (D) at (0,2) [circle,draw] {$4$};
\draw (A) -- (B) -- (C) -- (D) -- (A) -- (C);
\draw (D) -- (B);
\end{tikzpicture}

\end{tabular}
\end{center}
\caption{\label{fig1}The original Tower of Hanoi problem with $3$ pegs ($\mathrm{K}_3$) and $4$ pegs ($\mathrm{K}_4$).}
\end{figure}
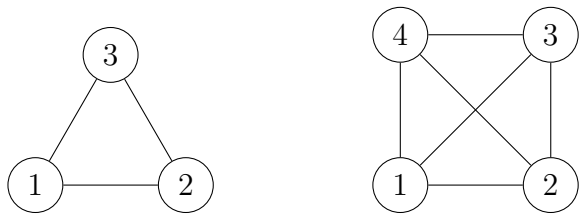

The main application of the generalized Frame-Stewart numbers is in giving upper bounds of the number of moves for the Tower of Hanoi problem on some simple graphs. For the Tower of Hanoi problem on the complete graph with $k\ge 3$ vertices and $n\ge 0$ disks, we retrieve the Frame-Stewart numbers $\St_k(n)$ stated in Section~1. In the sequel of this section, we consider other special cases where $G$ is the path graph $\mathrm{P}_3$ or the star graph $\mathrm{S}_k$.

\subsection{On the path graph $\mathrm{P}_3$}

The following theorem shows that the optimal number of moves for the Tower of Hanoi problem on the path graph $\mathrm{P}_3$ is given by the generalized Frame-Stewart numbers.

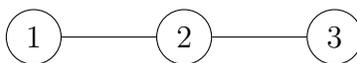
\begin{figure}[!h]
\begin{center}
\begin{tikzpicture}
\node (1) at (0,0) [circle,draw] {1};
\node (2) at (2,0) [circle,draw] {2};
\node (3) at (4,0) [circle,draw] {3};
\draw (1) -- (2) -- (3);
\end{tikzpicture}
\end{center}
\caption{\label{fig2}The path graph $\mathrm{P}_3$.}
\end{figure}

\begin{thm}\label{thm2}
Consider the Tower of Hanoi problem on $\mathrm{P}_3$, as depicted in Figure~\ref{fig2}. The minimum number of moves to transfer $n\geq1$ disks
\begin{itemize}
\item
from peg 1 to peg 3 is $\G_3(n)=2\cdot\sum_{i=0}^{n-1}3^i$, where $(p_3,q_3)=(3,2)$;
\item
from peg 1 to peg 2 is $\G_3^1(n)=\sum_{i=0}^{n-1}3^i$, where $(p_3,q_3)=(3,1)$.
\end{itemize}
\end{thm}

Though the fact of this theorem is rather well-known (e.g., see \cite{Sapir}), we present a short proof to see the connection with the generalized Frame-Stewart numbers.
\begin{proof}
We begin with the transfer between peg 1 and peg 3. In order to move the biggest disk from peg 1 to peg 3, we have to first move it from peg 1 to peg 2 and so the $n-1$ smallest disks must be on peg 3. The $n-1$ smallest disks are transferred from peg 1 to peg 3 in $\G_3(n-1)$ moves. Then, we move the biggest disk from peg 1 to peg 2. In order to move this disk to peg 3, we transfer the $n-1$ smallest disks from peg 3 to peg 1 in $\G_3(n-1)$ moves. Finally, we put the biggest disk from peg 2 to peg 3 in $1$ move and the $n-1$ smallest disks from peg 1 to peg 3 in $\G_3(n-1)$ moves. The total number of moves for $n$ disks is then $3\cdot\G_3(n-1)+2$, which corresponds to $\G_3(n)$ as announced. Since this is the best possible, $\G_3(n)$ is the optimal number of moves.
\par For the transfer between peg 1 and peg 2, as before, in order to move the biggest disk from peg 1 to peg 2, we have to first transfer the $n-1$ smallest disks from peg 1 to peg 3. As proved above, the minimum number of moves to do this is $\G_3(n-1)$. Moreover, we know that $\G_3(n-1)=2\cdot\G_3^1(n-1)$ by Proposition~\ref{prop2}. Then, after moving the biggest disk from peg 1 to peg 2, the $n-1$ smallest disks are transferred from peg 3 to peg 2. It is done in $\G_3^1(n-1)$ moves. Thus, we conclude that the minimum number of moves for transferring $n$ disks from peg 1 to peg 2 is $3\cdot\G_3^1(n-1)+1$ as announced.
\end{proof}

\subsection{On the star graph $\mathrm{S}_k$}

We end this section by considering the Tower of Hanoi problem on the star graph $\mathrm{S}_k$ with $k$ vertices and $k-1$ edges. For $k=3$, the graph $\mathrm{S}_3$ corresponds to the path graph $\mathrm{P}_3$. The star graphs for $k=4$ and $k=5$ are depicted in Figure~\ref{fig3}.

\begin{figure}[h!]
\begin{center}
\begin{tabular}{ccccc}

\begin{tikzpicture}

\node (1) at (0,1.15470054) [circle,draw] {1};
\node (2) at (0,3.46410162) [circle,draw] {2};
\node (3) at (-2,0) [circle,draw] {3};
\node (4) at (2,0) [circle,draw] {4};

\draw (2) -- (1) -- (3);
\draw (1) -- (4);

\end{tikzpicture}
& & & &
\begin{tikzpicture}

\node (1) at (0,2) [circle,draw] {1};
\node (2) at (-2,4) [circle,draw] {2};
\node (3) at (2,4) [circle,draw] {3};
\node (4) at (2,0) [circle,draw] {4};
\node (5) at (-2,0) [circle,draw] {5};

\draw (2) -- (1) -- (4);
\draw (3) -- (1) -- (5);

\end{tikzpicture}
\\
\end{tabular}
\end{center}
\caption{\label{fig3}The star graphs $\mathrm{S}_4$ and $\mathrm{S}_5$.}
\end{figure}
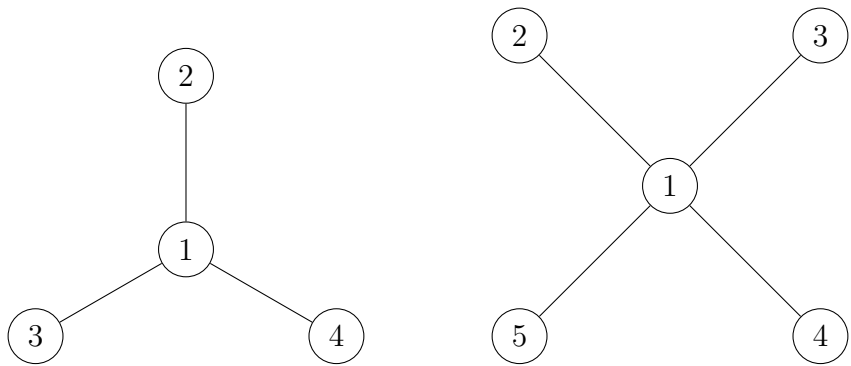

Stockmeyer \cite{Stock} considered the Tower of Hanoi problem on the star graph $\mathrm{S}_4$, where all the $n$ disks are transferred from one leaf of the graph to another leaf (for instance, the problem of transferring disks in the minimal number of moves from peg~2 to peg~3 in Figure~\ref{fig3}). He described a recursive algorithm which achieved a good (seemingly the best) upper bound; thus, called it the ``presumed optimal'' algorithm. Here, we generalize this algorithm to the star graph $\mathrm{S}_k$ for arbitrary $k \geq 3$ and show that disks can be transferred from one leaf to another in $\G_{k}(n)$ moves.

\begin{thm}\label{thm3}
Let $k\ge 3$ be an integer. Consider the Tower of Hanoi problem on the star graph $\mathrm{S}_k$ in which $n\geq 1$ disks are transferred from one leaf of the graph to another leaf. Then, an upper bound on the minimal number of moves to solve this problem is given by the generalized Frame-Stewart number $\G_{k}(n)$, where $(p_3, q_3) = (3,2)$ and $(p_i, q_i) = (2, 1)$ for $4 \leq i \leq k$.
\end{thm}

\begin{proof}
By induction on $k$ of $\mathrm{S}_k$. When $k=3$, as noted before, the star graph $\mathrm{S}_3$ corresponds to the path graph $\mathrm{P}_3$. So by Theorem~\ref{thm2}, $\G_3(n)$, where $(p_3, q_3) = (3, 2)$, is the minimum number of moves to transfer $n$ disks from peg 2 to peg 3. For $k\geq4$ and $n=1$, we can transfer one disk from peg~2 to peg~3 in only $\G_k(1)=2$ moves. Suppose now that the result is true for any number of disks up to $\mathrm{S}_{k-1}$ and until $n-1$ disks for $\mathrm{S}_k$. Then, $n$ disks are recursively transferred from peg 2 to peg 3 as follows. For some integer $t$ such that $1\leq t\leq n$,
\begin{itemize}
\item
transfer the $n-t$ smallest disks from peg $2$ to peg $k$ in $\G_{k}(n-t)$ moves;
\item
consider the remaining $k-1$ pegs and the subgraph obtained after deleting the vertex of peg $k$, which is the star graph $\mathrm{S}_{k-1}$, and transfer the $t$ largest disks from peg $2$ to peg $3$ in $\G_{k-1}(t)$ moves;
\item
transfer the $n-t$ smallest disks from peg $k$ to peg 3 in $\G_{k}(n-t)$ moves.
\end{itemize}
We choose the integer $t$ such that the number of moves $2\cdot\G_{k}(n-t)+\G_{k-1}(t)$ is minimized. 
Thus, the number of moves of this algorithm is 
$
\min_{1 \le t \le n}\bigl\{ 2\cdot \G_{k}(n-t) + \G_{k-1}(t) \bigr\}
$, 
which is, by the assumption of induction, equal to $\G_{k}(n)$ with $(p_3, q_3) = (3,2)$ and $(p_i, q_i) = (2, 1)$ for $4 \leq i \leq k$. 
This completes the proof of Theorem~\ref{thm3}.
\end{proof}

\section*{Acknowledgements}
The authors would like to thank the two anonymous reviewers for the time they spent on the reading of our manuscript and for the comments they posted, especially for the simpler proof of Proposition~\ref{prop3} that one of the reviewers suggested us.

\vspace{2cm}
\noindent\textbf{Authors' Address:}

\begin{itemize}

\item Jonathan Chappelon,\\
Laboratoire de Math\'ematiques Pures et Appliqu\'ees,\\
Universit\'e du Littoral C\^ote d'Opale\\
50 rue F. Buisson, B.P. 699,\\
F-62228 Calais Cedex, France\\
e-mail: jonathan.chappelon@lmpa.univ-littoral.fr

\item Akihiro Matsuura,\\
School of Science and Engineering,\\
Tokyo Denki University\\
Ishizaka, Hatoyama-cho, Hiki, 350-0394, Japan\\
e-mail: matsu@rd.dendai.ac.jp

\end{itemize}

\end{document}